%% file: salem_final.tex
\documentclass[11pt]{amsart} 

\usepackage{graphicx}
\usepackage[nice]{nicefrac}
\usepackage{amssymb}
\usepackage{epstopdf}
\usepackage{amsmath,amsthm,amscd,amssymb} 

\usepackage{latexsym}
\usepackage[colorlinks,citecolor=red,pagebackref,hypertexnames=false]{hyperref}

\usepackage[shortlabels]{enumitem} 
\usepackage{calc}
\setlist{itemsep=.125in}
\setlist[1]{topsep=.25in}
\setlist[1]{labelindent=\parindent}
\setlist[1]{leftmargin=\leftmargin - .125in}

\usepackage{tikz}
\usetikzlibrary{decorations.pathreplacing}

\DeclareMathOperator{\Mat}{Mat} 
\DeclareMathOperator{\Tr}{Tr} 
\DeclareMathOperator{\frob}{Frob} 

\newcommand{\matid}{\ensuremath{I}} 

\newcommand{\F}{\ensuremath{\mathbb{F}}} 
\newcommand{\C}{\ensuremath{\mathbb{C}}} 

\numberwithin{equation}{section}

\theoremstyle{plain}
\newtheorem{theorem}{Theorem}[section]
\newtheorem{lemma}[theorem]{Lemma}
\newtheorem{corollary}[theorem]{Corollary}
\newtheorem{proposition}[theorem]{Proposition}

\theoremstyle{definition}
\newtheorem{definition}[theorem]{Definition}

\newtheorem{example}[theorem]{Example}

\theoremstyle{remark}
\newtheorem{remark}[theorem]{Remark}

\title[\parbox{14cm}{\centering{Salem sets in modules over finite rings \hspace{1in}}} \quad]{The square root law and structure of finite rings}

\author{A. Iosevich}
\address{A. Iosevich\\ Department of Mathematics\\ University of Rochester\\ Rochester, NY\ 14627 }
\email{iosevich@math.rochester.edu}

\author{B. Murphy}
\address{B. Murphy \\ Department of Mathematics\\ University of Rochester \\ Rochester, NY\ 14627 }
\email{murphy@math.rochester.edu}

\author{J. Pakianathan}
\address{J. Pakianathan \\ Department of Mathematics\\ University of Rochester \\ Rochester, NY\ 14627 }
\email{}

\begin{document}
\maketitle

\begin{abstract} Let $R$ be a finite ring and define the hyperbola $H=\{(x,y) \in R \times R: xy=1 \}$.
Suppose that for a sequence of finite odd order rings of size tending to infinity, the following ``square root law'' bound holds with a constant $C>0$ for all non-trivial characters $\chi$ on $R^2$:
\[
\left| \sum_{(x,y)\in H}\chi(x,y)\right|\leq C\sqrt{|H|}.
\]
Then, with a finite number of exceptions, those rings are fields.

For rings of even order we show that there are other infinite families given by Boolean rings and Boolean twists which satisfy this square-root law behavior. We classify the extremal rings, those for which the left hand side of the expression above satisfies the worst possible estimate. We also describe applications of our results to problems in graph theory and geometric combinatorics. 

These results provide a quantitative connection between the square root law in number theory, Salem sets, Kloosterman sums, geometric combinatorics, and the arithmetic structure of the underlying rings.  
\end{abstract}         

\section{Introduction} 

\vskip.125in 

The square root law is a ubiquitous concept in modern mathematics. Roughly speaking, it says that 
$$ \left|\sum \ oscillating \ terms \ of \ modulus \ 1 \right| \leq C \sqrt{\# \ terms}.$$ 

\vskip.125in 

Many classical problems and open conjectures can be related to the square root law. %
For example, the Riemann hypothesis can be restated \cite{titchmarsh1986theory} in terms of the square root law applied to the exponential sum 
$$
\sum_{0<p \leq q \leq H; (p,q)=1} e^{2 \pi i \frac{p}{q}}.
$$ 

\vskip.125in 



Another example is provided by the Hardy Circle Conjecture, which says that the number of lattice points inside the disk of radius $R$ in the plane is equal to $\pi R^2$ plus an error term of size at most $C_{\epsilon}R^{\frac{1}{2}+\epsilon}$ for any $\epsilon>0$. Expressing the number of lattice points as an exponential sum once again puts this problem into the framework of the square root law.  

Even if we were to only stick to famous problems, the list of situations where the square root law comes into play is very large.
An interested reader can take a look at a very informative survey by Barry Mazur \cite{mazur2008finding} where several aspects of this concept are exposed.
The manifestation of the square root law that is most relevant to us is Deligne's proof \cite{deligne1974conjecture, deligne1980conjecture} of the Riemann hypothesis for finite fields.
See also \cite{kowalski2010aspects} for a very nice survey of the problem.
One of the key aspects of this theory is obtaining sharp bounds for Kloosterman type sums \cite{weil1948exponential}, in particular the bound 
$$
\left| \sum_{s \in {\Bbb F}_q^{*}} \chi(as+bs^{-1}) \right| \leq 2 \sqrt{q},
$$
where ${\Bbb F}_q$ is the finite field with $q$ elements, $(a,b) \in {\Bbb F}_q^2 \backslash \{(0,0)\}$, ${\Bbb F}_q^{*}$ is the field's multiplicative group and $\chi$ is a non-trivial additive character on ${\Bbb F}_q$. 

The square root law is discussed in this paper in the context of \emph{Salem sets}. 




\begin{definition} \label{salemdefinition} Let ${\{R_i\}}_{i=1}^{\infty}$ denote a set of finite rings such that $|R_i| \to \infty$ as $i \to \infty$\footnote{Here and throughout, if $S$ is a finite set, $|S|$ denotes the number of elements of $S$.}.
Let ${\{S_i \}}_{i=1}^{\infty}$ denote the collection of sets such that $S_i \subset R_i^d$, the $d$-dimensional module over $R_i$. Let $\widehat{S_i}(\gamma_i)$ denote the Fourier transform of $S_i$, viewed as the characteristic function of the set $S_i$\footnote{See section \ref{section: proof} for the definition of the Fourier transform.}.

\begin{enumerate}[(i)]
\item Suppose that for every $\epsilon>0$ there exists $C_{\epsilon}>0$, independent of $i$, such that 
\[
|\widehat{S}_i(\gamma_i)| \leq C_{\epsilon} {|R_i|}^{-d} {|S_i|}^{\frac{1}{2}+\epsilon}
\]
for every non-zero $\gamma_i \in \Gamma_i$.
Then we say that ${\{S_i\}}_{i=1}^{\infty}$ is \emph{Salem}\/ with respect to ${\{R_i^d \}}_{i=1}^{\infty}$. 
\item Suppose that there exists $C>0$ such that 
\[
|\widehat{S}_i(\gamma_i)| \leq C {|R_i|}^{-d} {|S_i|}^{\frac{1}{2}}
\]
for every non-zero $\gamma_i \in \Gamma_i$. 
\end{enumerate}
Then we say that ${\{S_i\}}_{i=1}^{\infty}$ is \emph{purely $C$-Salem}\/ with respect to ${\{R_i^d \}}_{i=1}^{\infty}$. 
\end{definition} 

The sphere provides us with a way to construct examples of Salem sets in a discrete setting. Let ${\Bbb F}_q$ denote the finite field with $q$ elements. Let 
$$ S_t=\left\{x \in {\Bbb F}_q^d: ||x|| \equiv x_1^2+\dots+x_d^2=t \right\},$$ where $t$ is a unit in ${\Bbb F}_q$. It is well-known (see for example \cite{iosevich2007erdos}) that 
\[
|\widehat{S}_t(\gamma)| \leq 2q^{-d} q^{\frac{d-1}{2}}
\]
if $\gamma $ is non-zero.
Since $|S_t|=q^{d-1}+$ lower order terms, we instantly recover the Salem property with $\epsilon=0$ for any sequence of fields. 
The proof in \cite{iosevich2007erdos} also shows that the hyperboloid is a pure Salem set for any sequence of fields, and a Gauss sum estimate shows the same for the paraboloid.

We have already seen that some explicitly defined sets such as the sphere, paraboloid and the hyperboloid are all Salem sets over finite fields ${\Bbb F}_q$.%
Does this phenomenon persist over more general rings?
The main thrust of this paper is that the answer is, in general, no, at least for odd order rings.
In the case of even rings, the main culprits are large Boolean rings and we classify the set of exceptions below. 

\vskip.125in 

\begin{theorem} \label{hyperbola} Let ${\{R_i \}}_{i=1}^{\infty}$ denote a sequence of odd order finite rings with $|R_i| \to \infty$. Let 
$$H_i=\{x \in R_i^2: x_1x_2=1 \}$$ denote the hyperbola with respect to the ring $R_i$. 

\vskip.125in 

Suppose that ${ \{H_i \}}_{i=1}^{\infty}$ is purely Salem with respect to ${ \{R^2_i \}}_{i=1}^{\infty}$. Then there exists $i_0$ such that for all $i>i_0$ $R_i$ is a field. \end{theorem} 

\vskip.125in 

Our proof examines rings of general (odd or even) order also and our results show that if one restricts to finite rings which have no $\mathbb{Z}/2\mathbb{Z}$-factors in their semisimple decomposition, or even a bounded number of such factors, then Theorem~\ref{hyperbola} still holds in the sense that if the hyperbola is purely Salem with respect to such a sequence rings, then all but finitely many of these rings are fields. 

In the presence of an unlimited number of $\mathbb{Z}/2\mathbb{Z}$ factors, we establish other sequences which have a purely Salem hyperbola. Specifically, we show that if 
$$R_n=\mathbb{Z}/2\mathbb{Z} \times \dots \times \mathbb{Z}/2\mathbb{Z}$$ is the Boolean ring of order $2^n$ then the sequence of these rings has a purely Salem hyperbola. More generally if $R$ is any fixed finite ring, the sequence $S_n=R \times R_n$ has a purely Salem hyperbola.

\vskip.125in 

Theorem 5.14 on page 14 of \cite{babai1989fourier} shows that a random construction yields Salem sets with respect to any sequence of rings $R_i$ of size tending to infinity.
However such constructions result in a logarithmic loss, which means that the resulting sequence is Salem, but not purely Salem (see Definition \ref{salemdefinition} above).
Salem sets share many properties with random sets, so it is interesting when a specific set of geometric and arithmetic importance like the hyperbola, sphere, or parabola exhibits Salem set behavior.

The Fourier coefficients of the hyperbola can be interpreted as generalized Kloosterman sums in the resulting rings as explained in Section~\ref{section: proof}. These sums are an important class of exponential sums with numerous number theoretic applications. In the process of proving Theorem~\ref{hyperbola} we introduce the concept of  the \emph{Kloosterman-Salem number}\/ of a finite ring $R$, denoted by $C_R$, which measures quantitatively how well the ``square-root law'' holds for Kloosterman sums over that ring.
More precisely, the Kloosterman-Salem number is the smallest positive number $C$ such that
$$
\left| \sum_{ x \in R^*} \chi_m(x)\chi_n(x^{-1})\right| \leq C\sqrt{|R^*|}
$$
for all $(m,n) \in R^2\setminus\{(0,0)\}$.

The larger this number, the weaker the form of the resulting square-root law. In the course of proving Theorem~\ref{hyperbola}, we show that for any threshold $\alpha \in (0, \infty)$, only a finite number of non-field odd order finite rings have Kloosterman-Salem number $C_R < \alpha$. On the other hand, results of Weil, Deligne, and Nicholas Katz show finite fields $F$ have Kloosterman-Salem number asymptotic to $2$ as $|F| \to \infty$. For details please refer to Section~\ref{section: proof} below. 

\vskip.125in 

The following theorem, interesting in its own right, summarizes our quantitative results on the Kloosterman-Salem numbers. 

\begin{theorem}
\label{thm: mainhyperbola}
Let $\alpha \in (0, \infty)$ be a threshold. Then
\smallskip

\begin{itemize} 
\item If $\alpha < 2$, only a finite number of odd order finite rings have Kloosterman-Salem number $C_R$ with $C_R \leq \alpha$.

\item If $\alpha > 2$, only a finite number of odd order finite rings which are not fields have $C_R \leq \alpha$. All but at most a finite number of finite fields have $C_R \leq \alpha$.

\item $\lim_{|F| \to \infty} C_F = 2$ where the limit is taken over any sequence of finite fields.

\item Every finite ring has $1 \leq C_R \leq \sqrt{|R^*|}$. The rings with $C_R=1$ are exactly the finite Boolean rings. 

\item The finite rings with $C_R = \sqrt{|R^*|}$ are called extremal rings. A finite field is extremal if and only if it has order $2, 3$ or $4$. For any finite ring $R$, $S=R \times B$, where $B$ is a nontrivial Boolean ring, has $S$ an extremal ring.

\item Every finite non-Boolean ring has $C_R \geq \sqrt{2}$. 

\end{itemize}
\end{theorem}
These results are illustrated in Figure \ref{fig:KSnum}.

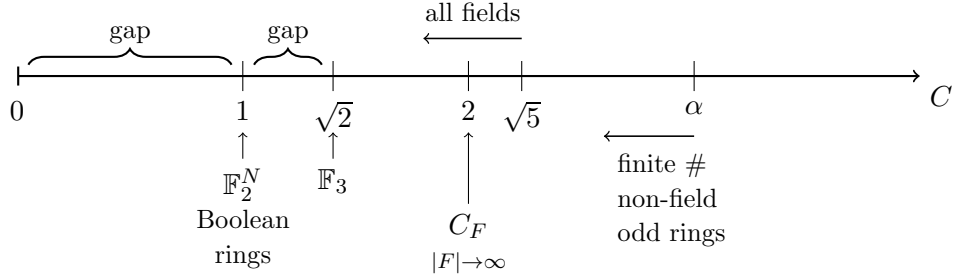
\begin{figure}[h]
  \centering
  \input{KSnum.tikz}
  \caption{Kloosterman-Salem numbers}
  \label{fig:KSnum}
\end{figure}




\vskip.125in 

\subsection{A graph theoretic viewpoint} The methods of this paper can be used to yield some light on Erd\H os-type problems in geometric combinatorics and related graph theoretic questions. This subsection is dedicated to a brief discussion of these topics. 

Let $R$ be a finite associative ring with a unit and define a graph, called the hyperbola graph of $R$, with vertices being the elements of $R^2$ and two vertices $x$ and $y$ are connected by an edge if $x-y \in H$, where $H$ is a hyperbola defined above. 

\begin{definition}
The {\it clique} number of a graph is the number of vertices in the largest complete subgraph.
The \emph{independence}\/ number of a graph is the number of vertices in the largest edgeless subgraph.
The {\it chromatic} number of a graph is the smallest number of distinct colors for the vertices such that if two vertices are connected by an edge, they are of different color.
\end{definition}

\begin{definition} The {\it spectrum} of a graph is the collection of eigenvalues of its adjacency matrix. A graph is regular if all its vertices have the same degree $d$ and for such a graph $d$ is the largest eigenvalue of the adjacency matrix, and indeed the spectrum lies in the interval $[-d,d]$. The value $-d$ is in the spectrum of a $d$-regular graph if 
and only if the graph is bipartite. For a non-bipartite graph, the {\it spectral gap} is the size of the gap between the largest and 2nd largest eigenvalue. The spectral gap is defined to be zero if and only if $d$ has multiplicity $\geq 2$ in the spectrum. \end{definition} 

\vskip.125in 

\begin{theorem} \label{chromatic} Let $R$ be an associative finite ring with a unit. Let $C_R$ denote the Kloosterman-Salem number of $R$. Then: 
\begin{itemize}
\item The hyperbola graph is a regular graph of degree $d=|R^*|$ and it is connected and not bipartite if and only if $C_R < \sqrt{|R^*|}$ i.e., the ring is not extremal. 
\item The spectrum of the hyperbola graph consists exactly of $|R|^2$ times the Fourier coefficients of the hyperbola's characteristic function. 
\item For a non-extremal ring, the spectral gap of the hyperbola graph is $|R^*|-C_R \sqrt{|R^*|}$. 
\item In the case that the hyperbola graph is connected and not bipartite (i.e. $R$ is non-extremal), a random walk on the graph is mixing i.e. converges to the uniform distribution at a rate determined by the spectral gap. More precisely for every starting node $i$, the probability $p_{ij}^t$ that after $t$ steps in a uniform random walk on the hyperbola graph, that we end up at vertex $j$ satisfies:
$$ \left|p_{ij}^t - \frac{1}{|R|^2} \right| \leq \left(\frac{C_R}{\sqrt{|R^*|}}\right)^t.$$
\item The independence number of the hyperbola graph of $R$ is at most $C_R|R|^2|R^*|^{-\frac 12}$.
\item The chromatic number of the hyperbola graph of $R$ is at least $\dfrac{{|R^{*}|}^{\frac{1}{2}}}{C_R}$. 
\item In particular, if $R$ is a finite field of order $q$, then the chromatic number of the hyperbola graph is at least $\frac{1}{2.14}\sqrt{q-1}$. 

\end{itemize}
\end{theorem} 
As an aside, note that any sequence of distinct finite rings with $C_R = o(\sqrt{|R^*|})$ yields a sequence of hyperboloid graphs with chromatic number tending to infinity.

\medbreak

\section{Proof of Theorem \ref{hyperbola}}
\label{section: proof}

\vskip.125in 

\subsection{Basic setup} 

Let $R$ be a finite ring, which is associative with identity but not necessarily commutative.
We view the hyperbola as $H=\times^{-1}(1)$ where $\times: R \times R \to R$ is the ring multiplication.
Thus 
$$
H=\{ (x,y) \in R^2: xy=1 \}.
$$ 
Clearly $|H|=|R^*|$ where $R^*$ is the unit group of $R$. 

\vskip.125in 

We now identify $R$ with its Pontryagin dual as its underlying additive group is finite abelian. We then identify the Pontryagin dual of $R^2$ with itself accordingly. The Haar measure is the counting measure normalized so that the  entire space has measure $1$. With this notation, the Fourier transform of the hyperbola $H  \subseteq R^2$ is given by  
$$ \hat{H}(m) = \frac{1}{|R|^2} \sum_{(x,y) \in R^2} H(x,y) \chi_m(-(x,y))
= \frac{1}{|R|^2} \sum_{x \in R^*} \chi_{m_1}(-x)\chi_{m_2}\left(-\frac{1}{x}\right),$$ where $\chi_m$ is the character in the dual group corresponding to $m \in R$ under the identification of $R$ with its dual, and $m=(m_1,m_2) \in R^2$. 

When the underlying abelian group of $R$ is cyclic, we can write $\chi_{m}(x)=\chi(mx)$ where $\chi$ is a fixed non-trivial character and this becomes the 
well-known Kloosterman sum
$$
\hat{H}(m) = \frac{1}{|R|^2} \sum_{x \in R^*} \chi \left(-m_1x -\frac{m_2}{x}\right)
$$
for $m=(m_1,m_2) \in R^2$. 

When $R$ is a finite field, it is well-known \cite{weil1948exponential} that 
\begin{equation}
  \label{eq: weil bound}
  |\hat{H}(m)| \leq 2 |R|^{-2}\sqrt{|R|} \quad\mbox{for $m\not=0$}.
\end{equation}
Thus the hyperbola is a pure $C$-Salem set  for some 
$$C \leq 2\sqrt{|R|/|R^*|}=\frac{2}{\sqrt{1-\frac{1}{|R|}}}$$ when $R$ is a finite field. 

\begin{definition}[Kloosterman-Salem number]
The \emph{Kloosterman-Salem number}\/ of $R$ is the infimum of numbers $C > 0$ such that
$$|\hat{H}(m)| \leq C|R|^{-2}|R^*|^{\frac{1}{2}}$$
for all $m \not=(0,0)$.
We denote this number $C_R$.
\end{definition}
 The Kloosterman-Salem number is clearly finite and non-negative for any finite ring.
Note that 
$$C_R=\frac{|R|^2}{\sqrt{|R^*|}} \max_{m \neq (0,0)} |\hat{H}(m)|.$$ 
If the Kloosterman-Salem number of $R$ is $C$, we say that $R$ is a \emph{pure $C$-Salem set}.

The vertical equidistribution of Kloosterman sums over finite fields, established by Nicholas Katz \cite{katz1988gauss}, implies that the constant 2 in Weil's bound is asymptotically sharp:
$$
\lim_{|F| \to \infty} C_F = 2,
$$
where the limit is taken over finite fields $F$.
We will establish that for any threshold $0 < \alpha < \infty$, there are only finitely many finite rings of odd order aside from fields with $C_R < \alpha$.
This means for thresholds $0 < \alpha < 2$ there are at most finitely many odd order finite rings with $C_R < \alpha$ whereas for thresholds $2 < \alpha < \infty$ almost all fields have $C_R < \alpha$ whereas only finitely many non field, odd order finite rings have $C_R < \alpha$.

In particular this means that any sequence $\{ R_n \}$ of distinct finite odd order rings with  Kloosterman-Salem number uniformly bounded, are eventually fields, in the sense that there exists $N>0$, such that for $n \geq N$, $R_n$ is a field.
This is the essence of Theorem \ref{hyperbola}. 

\subsection{A geometric criterion}
\label{geomcriterion}

An important set that encodes the connection between addition $+$ and multiplication $\times$ in the ring $R$ is given by 

\begin{eqnarray*}
N(R) &=& ( \times \circ - )^{-1} (1) \\ 
&=& (-)^{-1}(H) \\
&=& \{ (x,y) \in R^2 \times R^2: x-y \in H \} \\
&=& \{ (x,y) \in R^2 \times R^2: (x_1-y_1)(x_2-y_2)=1 \}
\end{eqnarray*}
where
$- : R^2 \times R^2 \to R^2$ is subtraction and $\times: R \times R \to R$ is multiplication.

\vskip.125in 

Similarly, we define 
$$N(E) = \{ (x,y) \in E \times E: x-y \in H \} = \{ (x,y) \in E \times E: (x_1-y_1)(x_2-y_2)=1 \}$$
and let $n(E) = |N(E)|$. 

\vskip.125in 

The next result relates the Kloosterman-Salem number to the size of the set $N(E)$.
\begin{theorem}
\label{thm: analyticalbound}
Let $R$ be a finite ring with Kloosterman-Salem number $C$. 
Then any set $E \subset R^2$ with $|E| > \frac{C|R|^2}{|R^*|^{\frac{1}{2}}}$ has $n(E) > 0$. More precisely, there exist $e_1, e_2 \in E$ 
such that $e_1-e_2 \in H$.
\end{theorem}
\begin{proof}
Let $q=|R|$. Then we have: 
\begin{eqnarray*}
n(E) &=& |\{ (x,y) \in E \times E : x-y \in H \}| \\
&=& \sum_{x,y} E(x)E(y)H(x-y) \\
&=& q^4 \sum_{m} |\hat{E}(m)|^2 \hat{H}(m) \\
&=& q^4 \sum_{m \neq 0} |\hat{E}(m)|^2 \hat{H}(m) + \frac{|E|^2|R^*|}{q^2} \\
&=& D(E) + \frac{|E|^2|R^*|}{q^2}
\end{eqnarray*}

where $D(E) = \sum_{m \neq 0} |\hat{E}(m)|^2 \hat{H}(m)$ is called the discrepancy of the set $E$ relative to the hyperbola $H$.

As $H$ is a pure $C$-Salem set of size $|R^*|$, we have 
$$
|D(E)| \leq \left( \sum_{m} |\hat{E}(m)|^2 \right) \cdot Cq^{-2}|R^*|^{\frac{1}{2}} = q^2|E|Cq^{-2}|R^*|^{\frac{1}{2}}
$$
where the last step follows by the Plancherel theorem.
Thus
$$|D(E)| \leq C|E||R^*|^{\frac{1}{2}}.$$

As $n(E)=D(E) + \frac{|E|^2|R^*|}{q^2}$ as long as $|D(E)| < \frac{|E|^2|R^*|}{q^2}$, we will have $n(E) > 0$. 
This is certainly the case when 
$$
C|E||R^*|^{\frac{1}{2}} < \frac{|E|^2|R^*|}{q^2}
$$
which happens when $|E| > \frac{Cq^2}{|R^*|^{\frac{1}{2}}}$. Thus Theorem \ref{thm: analyticalbound} is proven.

\end{proof}

\subsection{Bound on the size of ideals} 
\label{idealsize}

The sum-product formulation encapsulated in Theorem~\ref{thm: analyticalbound} leads directly to a bound on the size of proper ideals of $R$ in terms of the Kloosterman-Salem number $C$.

\begin{theorem}
\label{thm: idealbound}
Let $R$ be a finite ring with unit, with Kloosterman-Salem number $C$. Then any proper left (or right) ideal $I$ of $R$ has 
$$|I| \leq \frac{C|R|}{|R^*|^{\frac{1}{2}}}.$$
\end{theorem}

\begin{remark} As every ring (with unit) has at least the zero proper ideal, this in particular implies $C > 0$ for every finite ring. \end{remark}

\begin{proof}
We prove the theorem for proper left ideals. The proof for right ideals is similar. If $I$ is a proper left ideal then $E=R \times I \subseteq R^2$ has $n(E)=0$ 
as it is impossible to solve the equation $(x_1-y_1)(x_2-y_2)=1$ since $x_2-y_2 \in I$. Thus $|E|=|R||I| \leq \frac{C|R|^2}{|R^*|^{\frac{1}{2}}}$ by Theorem~\ref{thm: analyticalbound}. This completes the proof of Theorem \ref{thm: idealbound}. 
\end{proof}

We will see that the bound in Theorem~\ref{thm: idealbound} is sharp in the sense that for any fixed $0 < C < \infty$, only a limited class of finite rings satisfy it.

\subsection{Structure of finite rings}

Finite rings have a well studied structure, which we record here:
\begin{proposition}
  \label{algebra}
Let $R$ be a finite ring with unit.
\begin{enumerate}
\item $R$ contains a unique two-sided maximal ideal $J$, called the \emph{Jacobson radical}, such that $R/J$ is a semi-simple ring.
\item The semi-simple quotient $R/J$ is isomorphic to a product of matrix rings over finite fields.
\end{enumerate}
\end{proposition}
This was one of the first complete classification theorems in algebra.
We will sketch a proof, with references to Lang's \emph{Algebra}\/ \cite{lang2002algebra}, where the reader can find the details.
\begin{proof}
If $R$ is a finite ring with Jacobson radical $J$ then $R/J$ is finite and semisimple.
Since semisimple rings are direct products of simple rings (Chapter XVII, Theorem 4.4), it follows that $R/J$ is a finite product of simple rings.
Finally, every finite simple ring is a matrix ring over a finite field.

This last fact follows from two famous theorems.
First, as finite rings are Artinian, the Artin-Wedderburn theorem shows that finite simple rings are isomorphic to $\Mat_n(D)$, the $n \times n$ matrix ring over a finite division ring $D$.
It follows immediately from Wedderburn's theorem, which states that finite division rings are fields, that $D$ is a finite field.
\end{proof}

Now we will fix notation.
Let $R = \Mat_n(F)$ where $F$ is a finite field. Then $|R|=|F|^{n^2}$ and 
$$|R^*|=|GL_n(F)|=|F|^{n^2}\left(1-\frac{1}{|F|}\right)\left(1-\frac{1}{|F|^2}\right) \dots \left(1-\frac{1}{|F|^{n}}\right).$$ 
Thus 
$$\frac{|R^*|}{|R|} = \phi_R = \phi(n,|F|) = \left(1-\frac{1}{|F|}\right)\left(1-\frac{1}{|F|^2}\right) \dots \left(1 - \frac{1}{|F|^n}\right).$$
It will be convenient to have a uniform lower bound on $\phi(n,|F|)$.
  \begin{lemma}
    \label{lem: phi lower}
For all $n\geq 1$ and all finite fields $F$, we have $\phi(n,|F|)\geq 1/4$.
  \end{lemma}
  \begin{proof}
Since $\phi(n,|F|)=(1-\frac{1}{|F|})(1-\frac{1}{|F|^2}) \dots (1-\frac{1}{|F|^n})$ is increasing in $|F|$, the general case follows from the case $|F|=2$. 
As $\phi(n, 2)$ is monotonically decreasing as a function of $n$, it then suffices to establish that the infinite product
$$\alpha = \left(1-\frac{1}{2}\right)\left(1-\frac{1}{4}\right) \dots \left(1-\frac{1}{2^n}\right) \dots $$ is bounded below by $\frac{1}{4}$. A priori, by the monotone convergence theorem,  
$\alpha$ exists in the interval $[0,1)$. Taking logs and using the power series expansion for $\log(1-x)$ we get:
$$
-\log(\alpha) = \sum_{n=1}^{\infty} \sum_{k=1}^{\infty} \frac{1}{k2^{nk}}
$$
Exchanging the order of summation and summing the geometric series as a function of $n$, we get
$$
-\log(\alpha) = \sum_{k=1}^{\infty} \frac{\frac{1}{k2^k}}{1-\frac{1}{2^k}} =\sum_{k=1}^{\infty} \frac{1}{k(2^k-1)} \leq \sum_{k=1}^{\infty} \frac{1}{k} \left(\frac{1}{2}\right)^{k-1} = - 2\log\left(\frac{1}{2}\right)
$$
Thus $\log(\frac{1}{\alpha}) \leq \log(4)$ yielding $\alpha \geq \frac{1}{4}$ as desired.
  \end{proof}

If $R$ is a finite semisimple ring, then by Proposition \ref{algebra}
\begin{equation}
  \label{eq: semisimple}
  R=\Mat_{n_1}(F_1) \times \dots \times \Mat_{n_k}(F_k),
\end{equation}
where $F_1, \ldots, F_k$ are finite fields labelled so that $|F_1|^{n_1} \leq |F_2|^{n_2} \leq \dots \leq |F_k|^{n_k}$.
If $R$ is written as in \eqref{eq: semisimple}, we say that $R$ has $k$ \emph{semisimple factors}.
Note that
\[
R^*=GL_{n_1}(F_1)\times\cdots\times GL_{n_k}(F_k),
\]
and 
\[
|R^*|=|R|\phi(n_1,|F_1|)\cdots\phi(n_k,|F_k|).
\]

Finally, if $R$ is a finite ring with Jacobson radical $J$ then by Proposition \ref{algebra}, we have the short exact sequence of rings and ideals:
\[
0 \to J \to R \to \Mat_{n_1}(F_1) \times \dots \times \Mat_{n_k}(F_k) \to 0,
\]
where again $F_1, \dots, F_k$ are finite fields labelled so that $|F_1|^{n_1} \leq |F_2|^{n_2} \leq \dots \leq |F_k|^{n_k}$.

As the Jacobson radical has the property that if $a \in J$ then $1+a$ is a unit, it is easy to argue that the units of $R$ are exactly the elements that project to units of $R/J$ and so 
\[
|R^*| = |J||(R/J)^*|=|J||(R/J)|\phi(n_1,|F_1|)\dots\phi(n_k,|F_k|).
\]

\subsection{Finite simple rings}

We first establish some bounds in the restricted world of finite simple rings.
Specifically, we show the following.
\begin{proposition}
\label{pro: simplecase}
For a given threshold $\alpha>0$, all but finitely many finite simple rings with Kloosterman-Salem number less than $\alpha$ are fields or $2\times 2$ matrix rings over a field.
\end{proposition}

\begin{proof}[Proof of Proposition \ref{pro: simplecase}]
Let $M$ be the maximal left ideal in $\Mat_n(F)$ consisting of matrices with an all zero last column.
Then $|M|=|F|^{n^2-n}$ and so by Theorem~\ref{thm: idealbound} we have 
$$
|F|^{n^2-n} \leq \frac{C|F|^{n^2}}{\sqrt{|F|^{n^2}\phi_R(n, |F|)}}.
$$
Hence by lemma \ref{lem: phi lower}
\[
\frac 14 |F|^{n(n-2)}\leq |F|^{n(n-2)}\phi_R(n,|F|)\leq C^2.
\]
Since $|F|>1$, it is clear that $n$ is bounded.
If $n>2$, then
\[
|F|\leq (4C^2)^{1/n(n-2)}.
\]
Thus for any fixed $C$, there are finitely many choices of $n$, and for $n>2$ there are finitely many choice of $F$, which proves the theorem.
\end{proof}

\smallskip

\subsection{Kloosterman sums in matrix rings}
\label{matrix kloosterman}

We now will eliminate the case of $\Mat_2(F)$ in Proposition~\ref{pro: simplecase}.
By an explicit computation we will show that if $R=\Mat_2(F)$, then
\[
C_R \geq \frac{|F|-1+\frac{1}{|F|}}{\sqrt{(1-\frac{1}{|F|})(1-\frac{1}{|F|^2})}},
\]
which implies that only finitely many $2 \times 2$ matrix rings have a Kloosterman-Salem number less than a given threshold.
\begin{proposition}
\label{pro: matrix}
Let $\alpha \in (0, \infty)$ then there are only finitely many matrix rings $\Mat_2(F)$ with Kloosterman-Salem number less than $\alpha$. 
Thus there are only finitely many non-field, simple rings with Kloosterman-Salem number less than $\alpha$.
\end{proposition}

Let us first describe the Kloosterman sums arising from a matrix ring $\Mat_n(F)$.
Recall 
$$
\hat{H}(A,B) = \frac{1}{|F|^{2n^2}} \sum_{C \in GL_n(F)} \chi_A(-C)\chi_B(-C^{-1})
$$ 
for $(A,B) \in \Mat_n(F) \times \Mat_n(F)$. 
We identify $\Mat_n(F) \times \Mat_n(F)$ with its Pontryagin dual in the following specific way.
The trace form $(A,B) \to \Tr(AB)$ is $F$-bilinear,  symmetric and non-degenerate. As $\Mat_n(F)$ is a $F$-vector space, every irreducible character  of $\Mat_n(F)$ is the composition of a linear functional followed by a fixed nontrivial irreducible additive character $\chi$ of $F$, i.e., of the form  $\chi(L(x))$.
As the trace form is non-degenerate, every such functional can be taken of the form $L(-) = \Tr(-B)$ or $L(-)=\Tr(A-)$ for suitable $A, B \in \Mat_n(F)$.
Due to this we may choose an identification of $\Mat_n(F) \times \Mat_n(F)$ with its Pontryagin dual such that 
$$ \hat{H}(A,B) = \frac{1}{|F|^{2n^2}} \sum_{C \in GL_n(F)} \chi(-\Tr(CA + BC^{-1})) $$ and we shall do so from now on. 

The group $GL_n(F)$ acts on $\Mat_n(F) \times \Mat_n(F)$ by $D \cdot (A,B) = (DA,BD^{-1})$ and it is easy to check that $\hat{H}(A,B)$ is constant on orbits. More precisely, it is a 
$GL_n(F)$ invariant:
$$ \hat{H}(DA,DB)=\hat{H}(A,B). $$

Since the trace is a similarity invariant, conjugating by $C^{-1}$ in the defining expression shows that $\hat{H}$ is symmetric also i.e.
$$
\hat{H}(A,B)=\hat{H}(B,A).
$$
This invariance and symmetry makes the evaluation of $\hat{H}(A,B)$ reduce to a relatively decent number of cases based on the ranks of the matrices 
$A$ and $B$. The non degenerate case of rank $2$ matrix $A$ reduces as $\hat{H}(A,B)=\hat{H}(\matid, A^{-1}B)$ and the coefficients 
$$
\hat{H}(\matid,C) = \frac{1}{|R|^2} \sum_{D \in GL_2(F)} \chi(-\Tr(D + CD^{-1}))
$$
are probably the most interesting. However we will only use one particular degenerate coefficient in our arguments:
$$
\hat{H}\left(\begin{bmatrix} 1 & 0 \\ 0 & 0 \end{bmatrix}, \begin{bmatrix} 0 & 0 \\ 0 & -1 \end{bmatrix}\right) 
=\frac{1}{|R|^2} \sum_{D \in GL_2(F)} \chi \left(-a+\frac{a}{\Delta}\right)
$$
where we write $D=\begin{bmatrix} a & b \\ c & d \end{bmatrix}$ and $\Delta=\det(D)$, $R=\Mat_2(F)$.
To evaluate this we need to enumerate the distribution of upper-left entries $a$ and determinants $\Delta$ amongst the matrices in $GL_2(F)$.


Partition $GL_2(F)$ into the left cosets of $SL_2(F)$:
\[
GL_2(F) = \bigcup_{\Delta \in F^*} G_{\Delta},
\]
where $G_{\Delta}$ are the matrices with determinant $\Delta$. 
A simple computation shows that for $\Delta\not=0$, there are $p^2$ matrices in $G_\Delta$ with any given fixed nonzero $a$ as upper-left entry and $p(p-1)$ matrices in $G_\Delta$ with upper-left entry $a=0$.

For fixed $\Delta \in \mathbb{F}-\{0, 1\}$, we have 
$$
\sum_{D \in G_{\Delta}} \chi \left(-a+\frac{a}{\Delta}\right) =  |F|^2 \sum_{a \in \mathbb{F}^*} \chi \left(\left(-1 + \frac{1}{\Delta}\right)a\right) + |F|(|F|-1) = - |F|
$$
by character orthogonality applied to $\chi$ on $\mathbb{F}$.
On the other hand for $\Delta=1$ we get
$$
\sum_{D \in G_{1}} \chi \left(-a + \frac{a}{\Delta}\right) = |SL_2(F)|=(|F|-1)|F|(|F|+1).
$$

Putting everything together we get

\begin{eqnarray*}
\hat{H}\left(\begin{bmatrix} 1 & 0 \\ 0 & 0 \end{bmatrix}, \begin{bmatrix} 0 & 0 \\ 0 & -1 \end{bmatrix}\right)  &=& 
\frac{1}{|R|^2} \sum_{D \in GL_2(F)} \chi \left(-a+\frac{a}{\Delta}\right) \\ &=& 
\frac{1}{|R^2|} ((|F|-1)|F|(|F|+1) - (|F|-2)|F|).
\end{eqnarray*}

Thus 
$$
\frac{|R|^{2}}{\sqrt{|R^*|}} \left|\hat{H}\left(\begin{bmatrix} 1 & 0 \\ 0 & 0 \end{bmatrix}, \begin{bmatrix} 0 & 0 \\ 0 & -1 \end{bmatrix}\right)\right| =\frac{|F|(|F|^2-|F|+1)}{|F|^2\sqrt{(1-\frac{1}{|F|})(1-\frac{1}{|F|^2})}}
$$
and so the Kloosterman-Salem number of $R=\Mat_2(F)$ satisfies the claimed bound:
$$
C_R \geq \frac{|F|-1+\frac{1}{|F|}}{\sqrt{(1-\frac{1}{|F|})(1-\frac{1}{|F|^2})}}.
$$
As the right hand side goes to infinity as $|F| \to \infty$, we see there only finitely many finite fields $F$ such that the Kloosterman-Salem 
number of $\Mat_2(F)$ lies below any given threshold.
Together with Proposition~\ref{pro: simplecase}, this proves Proposition \ref{pro: matrix}.

\smallskip

\subsection{The semisimple case}

We will now extend the results of the previous sections to show that all but finitely many \emph{semisimple}\/ rings with no $\mathbb{F}_2=\Mat_1(\mathbb{F}_2)$ factors and Kloosterman-Salem number below a given threshold are fields.
\begin{proposition}
\label{pro: semisimple}
  For any fixed $0<\alpha<\infty$, all but finitely many finite semisimple rings $R$ with $C_R\leq\alpha$ and no $\mathbb{F}_2$-factors are finite fields. In 
  particular, for any fixed $0 < \alpha < \infty$, all but finitely many finite, odd order, semisimple rings $R$ with $C_R \leq \alpha$ are finite fields.
\end{proposition}

To prove Proposition \ref{pro: semisimple}, we need two lemmas.
Firstly we establish a useful general lower bound on the Kloosterman-Salem number of a finite ring.

\begin{proposition} Let $R$ be a finite ring with $1 \neq 0$ and let $C$ be its Kloosterman-Salem number. 
Then
$$
 \sqrt{\frac{1}{2}} < \sqrt{1 - \frac{|R^*|}{|R|^2}} \leq C \leq \sqrt{|R^*|}.
$$
Thus no finite ring has Kloosterman-Salem number $C \leq \sqrt{\frac{1}{2}}$.
\label{pro: lowerboundKS}
\end{proposition}
\begin{proof}
Let $H$ be the characteristic function of the hyperbola $\{ (x,y) \in R^2 | xy=1 \}$. Plancherel's Theorem gives
$$
|R|^2 \sum_{m \in R^2} |\hat{H}(m)|^2 = \sum_{x \in R^2} |H(x)|^2 = |H|=|R^*|.
$$
Thus
$$
|R|^2 \left(\frac{|R^*|^2}{|R|^4} + \sum_{m \neq 0} |\hat{H}(m)|^2 \right) = |R^*|
$$
so
$$
\frac{|R^*|}{|R|^2} - \frac{|R^*|^2}{|R|^4} \leq C^2 |R|^{-4} |R^*| (|R|^2 - 1) \leq C^2 |R|^{-2} |R^*|
$$
and hence
$$
C \geq \sqrt{1 - \frac{|R^*|}{|R|^2}} > \sqrt{1 - \frac{1}{|R|}} \geq \sqrt{\frac{1}{2}}. 
$$
As $|R|^2 \hat{H}(m)$ is the sum of $|R^*|$ terms of modulus one, it is clear $|\hat{H}(m)| \leq |R|^{-2} |R^*|$ which yields $C_R \leq \sqrt{|R^*|}$.

\end{proof}

We will see later that the lower bound for $C$ in Proposition~\ref{pro: lowerboundKS} can be strengthened to $1$ while the upper bound is sharp in general.

We will also need a formula for the Kloosterman-Salem number of a direct product of rings.
\begin{proposition} 
Let $R = R_1 \times R_2$ be a direct product of finite rings, then their Kloosterman-Salem numbers are related by
\[
C_R = \max(C_1|R_2^*|^{\frac{1}{2}}, |R_1^*|^{\frac{1}{2}}C_2),
\]
or equivalently
$$
\frac{C_R}{\sqrt{|R^*|}} = \max \left(\frac{C_1}{\sqrt{|R_1^*|}}, \frac{C_2}{\sqrt{|R_2^*|}}\right).
$$
\label{pro: directproducts} 
\end{proposition}
Note that because we require our rings to have a unit, and because the Kloosterman-Salem number is not defined when the dual group has no non-zero elements, this theorem only applies to non-trivial direct product decompositions.
\begin{proof}
First note that $R^* = R_1^* \times R_2^*$. 
As $R = R_1 \times R_2$ is also a decomposition of the underlying Abelian groups, the (irreducible) characters of $R$ are products of characters of $R_1$ and $R_2$.
The Hyperbola $H \subseteq R^2$ also decomposes as $H=H_1 \times H_2$ under the decomposition $R^2=R_1^2 \times R_2^2$.
Furthermore for any $m=(m_1,m_2) \in R_1^2 \times R_2^2=R^2$ it is easy to see that 
\begin{eqnarray*}
\hat{H}(m_1,m_2) &=& \frac{1}{|R|^2} \sum_{(x_1,x_2) \in R_1^* \times R_2^*} \chi_{m_1} \left(-\left(x_1,\frac{1}{x_1}\right)\right)\chi_{m_2}\left(-\left(x_2,\frac{1}{x_2}\right)\right) \\ &=&
\hat{H_1}(m_1) \hat{H_2}(m_2).
\end{eqnarray*}
The maximum of $|\hat{H_j}(m_j)|$ as $m_j$ varies over nonzero elements is by definition 
$$
C_{R_j} |R_j|^{-2}|R^*_j|^{\frac{1}{2}}
$$
while the value of  $|\hat{H_j}(0)|$ is $|R_j|^{-2}|R_j^*|$ for $j=1,2$.
Thus it is easy to calculate 
$$ \frac{|R|^2}{\sqrt{|R^*|}}\max_{(m_1,m_2) \neq (0,0)} |\hat{H}(m_1,m_2)| = 
\max(C_{R_1}C_{R_2}, C_{R_1}|R_2^*|^{\frac{1}{2}}, C_{R_2}|R_1^*|^{\frac{1}{2}} )$$
as claimed by considering the three cases $(m_1,m_2)$ both nonzero, $m_1=0$ and $m_2=0$. Using the trivial bound $C_{R_j} \leq \sqrt{|R_j^*|}$ 
shows that the maximum is one of the last two terms.
\end{proof}
Proposition~\ref{pro: directproducts} lets us construct examples to show that $\mathbb{Z}/2\mathbb{Z}$ factors have a limited effect on Kloosterman-Salem numbers and explains why we have to restrict to rings without these factors in this section.

\begin{example}[Boolean rings]
\label{ex: Boolean}
A finite Boolean ring is a direct product of finitely many $\mathbb{Z}/2\mathbb{Z}$'s.
Let $R_n= \mathbb{Z}/2\mathbb{Z} \times \dots \times \mathbb{Z}/2\mathbb{Z}$ be 
the Boolean ring of order $2^n$, which can be identified with the ring of $\mathbb{F}_2$-valued functions on a set of size $n$ under the usual operations of function addition and multiplication.
These rings have Kloosterman-Salem number $C=1$ independent of $n$ and hence give a sequence of rings $R_n$ with $|R_n| \to \infty$ such the Kloosterman-Salem number is uniformly bounded by $1$. 

To prove that $C=1$ for all finite Boolean rings first note that $C_{\mathbb{Z}/2\mathbb{Z}}=1$ by noting that  
the hyperbola consists of a single point $\{1,1\}$ and performing a quick calculation of $\hat{H}(m,n)$. We then use induction and the fact that $R_n = \mathbb{Z}/2\mathbb{Z} \times R_{n-1}$ in Proposition~\ref{pro: directproducts} to find 
\[
C_{R_n} = \max \left(1 \times 1, 1 \times \sqrt{|R_{n-1}^*|}, 1 \times \sqrt{|R_1^*|}\right)=1,
\] as all Boolean rings only have one unit, the vector $(1,1, \dots, 1)$. 

It is also easy to see that Boolean rings are exactly the rings with only one unit. To see this first note, that if a ring has exactly one unit then the Jacobson radical $J$ has $J=0$ 
as $|R^*|=|J||(R/J)^*| \geq |J|$. Thus $R$ is semisimple. By the Chinese remainder theorem, the simple matrix factors of $R$ must then also have only one unit.
It is easy then to see that they must be $\Mat_1(\mathbb{F}_2)=\mathbb{Z}/2\mathbb{Z}$ and so $R \cong \mathbb{Z}/2\mathbb{Z} \times \dots \times \mathbb{Z}/2\mathbb{Z}$ is a Boolean ring. 
\end{example}

\begin{example}[Twisting any ring by Boolean rings]
\label{ex: Boolean Twist}
Let $R$ be any finite ring, then let $S_n = \mathbb{Z}/2\mathbb{Z} \times \dots \times \mathbb{Z}/2\mathbb{Z} \times R$ be the direct product of 
$R$ with the Boolean ring of order $2^n$.

The product formula readily shows that 
$$C_{S_n} = \max \left( C_R, \sqrt{|R^*|}\right)=\sqrt{|R^*|}$$ is independent of $n$. Thus the sequence of rings $\{ S_n \}_{n=1}^{\infty}$ has $|S_n| \to \infty$ and uniformly bounded Kloosterman-Salem number. Despite this, these rings exhibit the worst square root law in the sense that $C_{S_n}=\sqrt{|R^*|}=\sqrt{|S_n^*|}$ achieves the general upper bound on the Kloosterman-Salem number given in Proposition~\ref{pro: lowerboundKS}.
\end{example}

Now we proceed with the proof of Proposition \ref{pro: semisimple}.
\begin{proof}[Proof of Proposition \ref{pro: semisimple}]
  Let $R$ be a finite semisimple ring as in equation \eqref{eq: semisimple} with no $\mathbb{Z}/2\mathbb{Z}$ factors.

If $R$ has only one semisimple factor, then $R$ is simple, so by Propositions \ref{pro: simplecase} and \ref{pro: matrix} all but finitely many such $R$ are fields.

Now suppose that $R$ has at least two semisimple factors, so that $R=\Mat_{n_1}(F_1)\times R_2$, where $R_2$ is a semisimple ring (note $|R_2| \geq 2$ as 
$1 \neq 0$ in our rings).
Let $C_1$ and $C_2$ denote the Kloosterman-Salem numbers of $\Mat_{n_1}(F_1)$ and $R_2$, respectively.
By Proposition \ref{pro: directproducts}, we have
\[
C_1|R_2^*|^{1/2}\leq C_R\qquad\mbox{and}\qquad C_2|GL_{n_1}(F_1)|^{1/2}\leq C_R.
\]
Bounding $C_1$ and $C_2$ below by Proposition \ref{pro: lowerboundKS} yields upper bounds for $|GL_{n_1}(F_1)|$ and $|R_2^*|$:
\[
|GL_{n_1}(F_1)|, |R_2^*|\leq 2C_R^2.
\]
If $R$ is a product of $k$ matrix rings, as in \eqref{eq: semisimple}, the previous equation implies that
\[
|GL_{n_2}(F_2)|\cdots|GL_{n_k}(F_k)|\leq 2C_R^2.
\]
This implies that
\[
\frac{1}{4} |F_j|^{n_j^2}\leq |GL_{n_j}(F_j)|\leq 2C_R^2
\]
for $j=1,\ldots,k$, so $n_j$ and $|F_j|$ are bounded for all $j$.
Further, as $F_j\not=\mathbb{Z}/2\mathbb{Z}$, we have $|GL_{n_j}(F_j)|\geq 2$, hence
\[
2^{k-1}\leq|GL_{n_2}(F_2)|\cdots|GL_{n_k}(F_k)|\leq 2C_R^2,
\]
which shows that $k$ is bounded.
As the size and number of $R$'s semisimple factors are bounded in terms of $C_R$, it follows that $|R|$ is bounded in terms of $C_R$.

It follows that there are finitely many semisimple $R$ with no $\mathbb{Z}/2\mathbb{Z}$ factors and with more than one semisimple factor and Kloosterman-Salem number $C_R\leq \alpha$, which concludes the proof.
\end{proof}
\begin{remark}
The same proof shows that all but finitely many semi-simple rings $R$ with $C_R\leq\alpha$ and a \emph{bounded number $\F_2$ factors}\/ (say $\leq n$ such factors) are fields.
\end{remark}

\subsection{Finite rings with Jacobson radical}

In this section we show that at most finitely many (odd) rings with Kloosterman-Salem number less than $\alpha>0$ have a non-zero Jacobson radical.

Firstly, we show that a ring with a non-zero Jacobson radical has a larger Kloosterman-Salem number than its semisimple part.
\begin{lemma}
\label{lem: pullback}
  Let $R$ be a finite ring with Jacobson radical $J$, and let $S=R/J$ be the semisimple part of $R$, so that we have a short exact sequence of rings and ideals
\[
0 \to J \to R \to S \to 0.
\]
If $C_R$ and $C_S=C_{R/J}$ denote the Kloosterman-Salem numbers of $R$ and $S$, then
\[
C_R\geq C_S|J|^{1/2}=C_{R/J}|J|^{1/2}.
\]
\end{lemma}
\begin{proof}
Let $\chi_m, \chi_n$ be any additive characters of $S$ (at least one of them nontrivial), pulling them back under the quotient map $\pi: R \to S$, one can view them as additive characters of $R$ which are equal to $1$ on $J$. Using these characters one obtains certain Kloosterman sums
$$
\hat{H}(m,n) = \frac{1}{|R|^2} \sum_{x \in R^*} \chi_m(-x)\chi_n(-\frac{1}{x}),
$$
which represent certain Fourier coefficients for the hyperbola of $R$.
As $\chi_m(-x)$ and $\chi_n(-\frac{1}{x})$ only depend on the image of $x, \frac{1}{x}$ in $S$, 
using that $R^*=\pi^{-1}(S^*)$, this sum degenerates into
$$
\hat{H}(m, n) = \frac{|J|}{|R|^2} \sum_{x \in S^*} \chi_m(-x)\chi_n(-\frac{1}{x}).
$$
Taking the maximum over $(m,n) \neq (0,0) \in S \times S$ one gets 
$$
\max_{(m,n) \in S \times S} |\hat{H}(m,n)| = \frac{|J|}{|R|^2} C_S |S^*|^{\frac{1}{2}}
$$
and so
$$
C_R |R|^{-2}|R^*|^{\frac{1}{2}} \geq \frac{|J|}{|R|^2} C_S |S^*|^{\frac{1}{2}}.
$$
(Note this last inequality is not necessarily an equality as the previous maximum was only over characters of $R$ induced from $S$ and not all characters of $R$.) Using that $|R^*|=|S^*||J|$ this simplifies to give
$$ C_R\geq C_S|J|^{1/2}=C_{R/J}|J|^{1/2}, $$  as was to be shown.
\end{proof}

We are now ready to show that there are finitely many finite rings with no $\mathbb{Z}/2\mathbb{Z}$ semisimple factors that have a non-zero Jacobson radical and a Kloosterman-Salem number below a given threshold. This is the last step in the proof of the main theorem.
\begin{proof}[Proof of $J\not=0$ case]
Let $R$ be a finite ring with Kloosterman-Salem number $C_R$ bounded above by some threshold $\alpha \in (0, \infty)$; further, assume that $R$ has no $\mathbb{Z}/2\mathbb{Z}$ semisimple facters.

It follows from Lemma \ref{lem: pullback} and Proposition~\ref{pro: lowerboundtwo} that $|J|\leq \alpha^2$ and $C_{R/J}\leq\alpha$.
As $R/J$ is semisimple and $C_{R/J}$ is bounded by $\alpha$, Proposition \ref{pro: semisimple} implies that $R/J$ must be a field in all but finitely many cases.
Thus we may assume that we have a short exact sequence
\[
0\to J\to R\to F\to 0,
\]
where $F$ is a finite field.
It remains to be shown that the Jacobson radical is zero in all but finitely many cases.

If $J \neq 0$, then $J/J^2 \neq 0$ by Nakayama's lemma.
As $J/J^2$ is a non-zero $R/J=F$-vector space, it follows that $|F|$ divides $|J|$.
Since $|J|$ is bounded by $\alpha^2$, it follows that $|F|$ is also bounded by $\alpha^2$, hence $|R|$ is bounded by $\alpha^4$, which implies that only a finite number of rings satisfy these conditions.  
\end{proof}

Thus we have established Theorem \label{thm: mainhyperbola} and in view of the above, Theorem \ref{hyperbola} follows. 

\smallskip

\section{Quantitative Results}

The proofs of Propositions \ref{pro: simplecase} and \ref{pro: matrix} yield explicit bounds which we record in the following corollary.
\begin{corollary}
\label{cor: explicitbounds}
Let $F$ be a finite field, let $R$ be the finite simple ring $\Mat_n(F)$, and let $C_R$ be the Kloosterman-Salem number of $R$.

If $n=2$,
\begin{equation}
  \label{eq:3}
  C_R\geq |F|-1+\frac 1{|F|}.
\end{equation}
If $n\geq 3$,
\begin{equation}
  \label{eq:4}
  C_R\geq\frac 12 |F|^{\frac{n(n-2)}2}.
\end{equation}
In particular, $C_R\geq\sqrt 2$ for all $F$ and all $n$.
In addition, we have
\begin{equation}
  \label{eq:5}
  n \leq \sqrt{2\log_2(C_R)}+2.
\end{equation}
\end{corollary}
\begin{proof}
If $n=2$, then by Proposition \ref{pro: matrix},
\[
C_R\geq \frac{|F|-1+\frac 1{|F|}}{\sqrt{\phi(2,F)}}\geq |F|-1+\frac 1{|F|},
\]
which proves \eqref{eq:3}.

Now suppose that $n\geq 3$.
In this case the proof of Proposition \ref{pro: simplecase} shows that
\begin{equation}
  \label{eq:1}
  \frac 14 |F|^{n(n-2)}\leq |F|^{n(n-2)}\phi(n,|F|)\leq C_R^2.
\end{equation}
Taking square roots proves \eqref{eq:4}.

To prove the bound \eqref{eq:5} on $n$, we combine trivial bound $2\leq |F|$ with \eqref{eq:1}:
\[
2^{n(n-2)-2}\leq C^2.
\]
For $n\geq 3$, we have $(n-2)^2\leq n(n-2)-2$, so
\[
2^{(n-2)^2}\leq C^2,
\]
which implies \eqref{eq:5} by taking logarithms.
For $n=2$, the upper bound is trivial and so the proof is complete.
\end{proof}

Using Lemma~\ref{lem: pullback} we can strengthen the general lower bound for Kloosterman-Salem numbers obtained in Proposition~\ref{pro: lowerboundKS}.
\begin{proposition}
\label{pro: lowerboundtwo}
Let $R$ be a finite ring with Kloosterman-Salem number $C_R$.
Then $C_R \geq 1$ with equality if and only if $R$ is a finite Boolean ring.
Any non-Boolean finite ring has $C_R \geq \sqrt{2}$.
\end{proposition}
\begin{proof}
Since $C_R \geq C_{R/J} |J|^{\frac{1}{2}}$, Lemma~\ref{lem: pullback} implies that it is enough to consider the semisimple case $J=0$.
The product formula in Proposition~\ref{pro: directproducts} reduces to the case where $R=\Mat_n(F)$.
By Corollary~\ref{cor: explicitbounds}, if $n\geq 2$ then $C_R\geq\sqrt 2$.

Thus we are reduced to the case $R=\Mat_1(F)=F$ a finite field.
Here by a result on Kloosterman sums, which we cite below, we have $C_F \geq \sqrt{2}$
for all finite fields besides $\F_2$.
A simple direct computation shows $C_{\F_2}=1$ as explained in Example~\ref{ex: Boolean}.

Furthermore it is easy to see from these arguments that $C_R=1$ if and only if $R$ is semisimple with all simple factors $\F_2$, a Boolean ring, and otherwise $C_R\geq\sqrt 2$.
\end{proof}
The lower bound $C_F\geq \sqrt 2$, where $F$ is a finite field, is implicit in the original work of Kloosterman \cite{kloosterman1927representation}.
A modern proof can be found on page 22 of \cite{kowalskiexponential}, where it is shown that if $|F|=q$, then
\begin{equation}
  \label{eq:2}
  C^2_F\geq\frac{2q^3-3q^2-3q-1}{(q-1)(q^2-q-1)}.
\end{equation}
For $q>3$, the right hand side of \eqref{eq:2} is greater than $2$, and one may easily check that $C_{\F_3}=\sqrt 2$.

\section{Extremal Rings}

In this section we study rings that have the worst possible square root law for Kloosterman sums.
We call such rings extremal:
\begin{definition}
A finite ring $R$ is \emph{extremal}\/ if its Kloosterman-Salem number $C_R$ achieves the general upper bound $C_R = \sqrt{|R^*|}$.
\end{definition}
Example \ref{ex: Boolean} shows that Boolean rings $\F_2^N$ are extremal, and example \ref{ex: Boolean Twist} shows that we can create extremal rings by taking products with Boolean rings.
It turns out that there are further examples, which we will partially classify.

We begin by providing an alternate characterization of extremal rings.
\begin{theorem}
\label{thm: extremal}
Let $R$ be a finite ring with Kloosterman-Salem number $C$. Then the following are equivalent:
\begin{enumerate}
\item $C < \sqrt{|R^*|}$.
\item $R$ is \emph{not}\/ an extremal ring.
\item The subset $\{ (x-1, x^{-1}-1):  x \in R^* \}$ of $R^2$ generates $R^2$ as an additive group.
\item For every $A, B \in R$, there exists a positive integer $n$ and units $x_1, \dots, x_n \in R^*$ such that
  \begin{align*}
    x_1 + \dots + x_n - n &= A \\
x_1^{-1} + \dots + x_n^{-1} - n &= B
  \end{align*}
\end{enumerate}
\end{theorem}

\begin{proof}
Note that (1) and (2) are equivalent by definition, and (3) and (4) are equivalent as $R \times R$ is finite and so a subset $S \subseteq R^2$ generates it as an additive group if and only if it generates it as a semigroup.
In this case this means any $(A,B) \in R^2$ is a finite sum of elements of the form $(x_i, x_i^{-1})$. 
Thus it remains to prove the equivalence of (2) and (3).

Suppose $R$ is an extremal ring, which means that $C_R = \sqrt{|R|^*}$.
This happens if and only if there exists $(m,n) \neq (0,0)$ such that
$$
|K(m,n)| = \left|\sum_{x \in R^*} \chi_m(x)\chi_n(x^{-1})\right|=|R^*|.
$$
That is, $K(m,n)$ has no cancellation.
As $|\chi_m(x)\chi_n(x^{-1})|=1$, this can only happen if $\chi_m(x)\chi_n(x^{-1})$ is constant for $x\in R^*$.

Since $\chi=\chi_m \otimes \chi_n : R \times R \to \C$ is a non-trivial additive character, its kernel $K$ is a proper subgroup of $R \times R$.
Elements in $R \times R$ have the same $\chi$-value if and only if they lie in the same coset of $K$, and so the Kloosterman sum $K(m,n)$ has no cancellation only if the hyperbola $H=\{ (x, x^{-1}) | x \in R^*\}$ lies in a single coset of $K$. 

Since $H-(1,1)$ is contained in $K-(1,1)$, it is clear that if $H-(1,1)$ cannot generate $R^2$ under addition.
Conversely, if $H-(1,1)$ generates a proper subgroup $K$ of $R^2$ then a pullback character under $\pi: R \times R \to (R \times R) /K$ yields a non-trivial additive character $\chi_m \otimes \chi_n$ of $R \times R$ for which $K(m,n)$ has no cancellation.
\end{proof}

\begin{corollary}
\label{cor: reducetosemisimple}
If $R$ is a finite ring with Jacobson radical $J$ then if the semisimple ring $R/J$ is extremal, this implies $R$ itself is extremal.
\end{corollary}
\begin{proof}
First recall the hyperbola $H$ in $R \times R$ maps onto the hyperbola $\bar{H}$ in $R/J \times R/J$ under the quotient map.
Thus if $H$ generates $R \times R$ as an additive group, $\bar{H}$ will generate $R/J \times R/J$ as an additive group.
Thus by Theorem~\ref{thm: extremal}, $R$ not extremal implies $R/J$ is not extremal.
The result follows by taking the contrapositive.
\end{proof}

Corollary~\ref{cor: reducetosemisimple} reduces questions about extremal rings to questions about extremal semisimple rings.

\begin{corollary}
\label{cor: extremalproducts}
If $R = R_1 \times \dots \times R_n$ is a direct product of finite rings then $R$ is extremal if and only if at least one of the $R_i, 1 \leq i \leq n$ is extremal.
\end{corollary}
\begin{proof}
One can prove this either using the product formula for Kloosterman-Salem numbers or by noting that the hyperbola $H$ in $R \times R$ is the direct product 
of the hyperbolas $H_i$ in $R_i \times R_i$. Thus $H-\{(1,1)\}$ generates $R \times R$ as an additive group if and only if each $H_i-\{(1_i,1_i)\}$ generates $R_i \times R_i$ as an additive group. Thus $R$ is not extremal if and only if all the $R_i$'s are not extremal. 
\end{proof}

Corollary~\ref{cor: extremalproducts} lets us reduce the questions about extremal semisimple rings to ones about extremal simple rings, i.e., $\Mat_n(F)$ 
where $F$ is a finite field. We deal with fields next.

\begin{proposition}
\label{pro: finite field}
Let $\F_q$ be the finite field of order $q$. Then $\F_q$ is extremal if and only if $q=2,3,4$.
\end{proposition}
\begin{proof}
Let $C$ be the Koosterman-Salem number of $\F_q$.
By the Weil bound \eqref{eq: weil bound}, any nontrivial Kloosterman sum is bounded by $2\sqrt{q}$. 
Thus the field is \emph{not}\/ extremal as long as $2\sqrt{q} < q-1$ as the Kloosterman sums are sums of $q-1$ elements of modulus one.
This is the case if $q^2-6q+1 > 0$ which holds as long as $q > 5$. Thus any finite field of size $q > 5$ is not extremal.

The field $\mathbb{F}_2$ has $C=1$ and only one unit so it is extremal.
The field $\mathbb{F}_3$ has $C=\sqrt{2}=\sqrt{|\F_3^*|}$ so it is extremal. 

For $\mathbb{F}_5$ the Kloosterman sums are given by
$K(m,n) = \sum_{x \in \F_5^*} \chi(mx + \frac{n}{x})$.
The $x=1,-1$ terms and $x=2,-2$ terms are complex conjugates and with a bit of calculation, we get
$$K(m,n) = 2\cos(2\pi (m+n)/5) + 2\cos(4\pi (m-n)/5).$$

As it is impossible to have $m+n=0=m-n$ without $m=n=0$ in $\mathbb{F}_5$ we see that  
$\mathbb{F}_5$ is not extremal.

Finally write $\F_4=\F_2[u]$ where $u$ is a primitive third root of unity and hence solves $u^2+u+1=0$.
Recall the trace $\Tr\colon\F_4 \to \F_2$ is given by $\Tr(a+bu) = (a+bu) + (a+bu^2)=2a+b(u+u^2)=b$ for any $a,b \in \mathbb{F}_2$, as the Galois group of $\F_4$ over $\F_2$ is cyclic of order two generated by the Frobenius map $\frob\colon x \to x^2$.
The Kloosterman sum is then given by 
$$
K(m,n) = \sum_{x \in \F_4^*} \chi \left(\Tr \left(mx + \frac{n}{x}\right)\right)
$$
where $\chi(s) = e^{\pi i x}$ is the nontrivial additive character of $\F_2$.
Thus 
$$K(m,n) = \chi(\Tr(m+n)) + \chi(\Tr(mu + n(1+u))) + \chi(\Tr(m(1+u) + nu)).$$ 

It follows that $K(1,1)=\chi(0)+\chi(0) + \chi(0)=3=|\F_4^*|$ and so $\mathbb{F}_4$ is extremal. 
\end{proof}

\section{Hyperbola Graphs}

Let $R$ be a finite ring and let $S$ be a subset of $R^d$ for some $d \geq 1$. We say $S$ is symmetric if $x \in S$ implies $-x \in S$. 
Please consult \cite{lovasz2007eigenvalues} for the graph theoretic background needed in this section.
\begin{definition}
Given a symmetric set $S \subseteq R^d$ for some $d \geq 1$. We define the \emph{$S$-graph $G_S$}\/ to be the graph whose vertex set is $V=R^d$ and where 
$v_1$ and $v_2$ are joined by a single edge in $S$ if and only if $v_1 - v_2 \in S$. 
\end{definition}
Note this graph has no multiple edges, and has loops if and only if $0 \in S$. It is a regular graph where each vertex has degree $d = |S|$. 
Graphs of these sort have been studied extensively \cite{chung1989quasi-random, chung1992quasi-random}.

Recall the adjacency matrix $\mathbb{A}$ of this graph is a $|V| \times |V|$ matrix whose rows and columns are indexed by the vertices of the graph and 
where $a_{ij} = 1$ if vertex $v_i$ is joined to vertex $v_j$ by an edge and $a_{ij}=0$ if not. 

We first relate the spectrum of the graph $G_S$, the set of eigenvalues of $\mathbb{A}$, to the Fourier coefficients of the characteristic function of the set 
$S$. 
\begin{proposition}
Let $G_S$ be the $S$-graph of a symmetric set $S \subseteq R^d$ and let $\mathbb{A}$ be its adjacency matrix. The eigenvectors of $\mathbb{A}$ 
are exactly the characters of the additive group of the ring $R$ and the character $\chi_{m}$ corresponds to eigenvalue $|R|^d \hat{S}(m)$, 
where $\hat{S}(m)$ is the Fourier coefficient of the characteristic function $S$ with respect to that character. Thus the spectrum of $\mathbb{A}$ is the same as the set of Fourier 
coefficients of $S$ scaled by $|R^d|$. In particular the spectral gap between the largest eigenvalue and one of 2nd largest magnitude is 
$$
|S| - \max_{m \neq 0} |R^d||\hat{S}(m)|.
$$
\end{proposition}
\begin{proof}
First note that we may think of a function $f: V=R^d \to \mathbb{C}$ as a column vector whose entries are indexed by the vertex set $V=R^d$ and whose $v$-th entry is $f(v)$. Under this identification, it is easy to check that the adjacency matrix $\mathbb{A}$ corresponds to an operator $g = A f$ where 
$$
g(v) = \sum_{u \in S} f(v+u).
$$
Now let $f=\chi_m$ be an additive character of $R^d$, then 
\begin{eqnarray*}
Af (v) &=& \sum_{u \in S} \chi_m(v+u) \\ &=& \sum_{u \in R^d} \chi_m(v)\chi_m(u)S(u) \\ &=& \chi_m(v) \sum_{u \in R^d} \chi_m(-u)S(u) \\ &=& |R^d| \hat{S}(m)\chi_m(v)
\end{eqnarray*}
for all $v \in V$. Thus $Af=|R|^d \hat{S}(m) f$ and $f=\chi_m$ is an eigenvector of $A$ with eigenvalue $|R|^d\hat{S}(m)$.
As $(R^d,+)$ is a finite abelian group, the number of such characters is $|R^d|=|V|$. As irreducible characters of finite groups are linearly independent, 
we see that we have indeed found all the eigenvectors of $\mathbb{A}$. The proposition follows.
\end{proof}

In a regular graph of degree $d$, $d$ is the largest eigenvalue of $\mathbb{A}$. It is also an eigenvalue of maximal magnitude though $-d$ is also in the spectrum and of equal magnitude if the graph is bipartite. Furthermore by a theorem of Frobenius, the multiplicity of $d$ as an eigenvalue of 
$\mathbb{A}$ is the same as the number of connected components of the graph. Thus we have the following corollary:

\begin{corollary}
\label{cor: nonextremal hyperbola graphs}
Let $G_S$ be the $S$-graph arising from a symmetric set $S \subseteq R^d$. Then $G_S$ is connected if and only if 
$$
\max_{m \neq 0} \hat{S}(m) < |R|^{-d}|S|.
$$
Furthermore we have 
$$
\max_{m \neq 0} |\hat{S}(m)| < |R|^{-d}|S|
$$
if and only if the graph is connected and not bipartite.
\end{corollary}
\begin{proof}
The first part follows from the Theorem of Frobenius mentioned in the preceding paragraph. The second part then follows as the only element of the spectrum 
that can have the same magnitude as $d$ besides $d$ itself is $-d$ and $-d$ is in the spectrum of a connected regular graph if and only if the graph is bipartite.
\end{proof}

\begin{corollary}
If $R$ is a finite ring which is not extremal then for any $A, B \in R$ there exists $n \geq 1$ and units $u_1, \dots, u_n$ such that
\begin{align*}
A &= u_1 + \dots + u_n \\
B &= u_1^{-1} + \dots + u_n^{-1}.
\end{align*}
\end{corollary}
\begin{proof}
If $R$ is not extremal, by Corollary~\ref{cor: nonextremal hyperbola graphs} we have that the hyperbola graph is a connected graph. 
Thus in particular it is possible to get from vertex $(0,0)$ to vertex $(A,B)$ with a simple path. This means that 
$(A,B) = (0,0) + (u_1, u_1^{-1}) + \dots + (u_n, u_n^{-1})$ for some $(u_j, u_j^{-1})$ on the hyperbola. This gives the result.
\end{proof}

\begin{definition}
Let $R$ be a finite ring and let $H \subseteq R^2$ be the hyperbola $H=\{ (u, u^{-1}) | u \in R^* \}$. The hyperbola graph is the graph arising from the symmetric set $H$.
By the earlier results of this section, this graph is regular of degree $d=|R^*|$ and has a spectrum given by $|R|^2$ times the Fourier coefficients of $H$. 
\end{definition}

\begin{corollary}
Let $R$ be a finite ring and $C$ be its Kloosterman-Salem number. Then if $R$ is not extremal, the hyperbola graph $G_H$ is connected and not 
bipartite. Furthermore the spectral gap is given by $|R^*| - \sqrt{|R^*|}C$.  Conversely when $R$ is extremal, the hyperbola graph $G_H$ is either 
disconnected or connected and bipartite.
\end{corollary}
\begin{proof}
The spectrum of a regular graph of degree $d$ is real and contained in the interval $[-d,d]$. It is connected if and only if $d$ has multiplicity $1$ as an 
eigenvalue and bipartite if and only if $-d$ is an eigenvalue.

When $R$ is not extremal, $|R^2| |\hat{H}(m)| < |R^*|$ for $m \neq 0$ and so $d=|R^*|$ has multiplicity one as an eigenvalue and $-d$ does not occur 
as an eigenvalue. Furthermore by definition 
$$|R|^2 \max_{m \neq 0} |\hat{H}(m)| = C\sqrt{|R^*|},$$ so the spectral gap of the hyperbola graph is given by 
$$|R^*| - C\sqrt{|R^*|}$$ and the corollary follows.
\end{proof}

\begin{example}[Hyperbola graphs of extremal examples]
Let $K_n$ denote the complete graph on $n$ vertices.
The hyperbola graphs of the extremal rings $\F_2,\F_3,$ and $\F_4$ are disjoint unions of complete graphs:
\begin{itemize}
\item The hyperbola graph of $\F_2$ is the disjoint union of two edges, that is, two $K_2$ graphs. 
\item The hyperbola graph of $\F_3$  is the disjoint union of $3$ triangles, that is, three $K_3$ graphs. 
\item The hyperbola graph of $\F_4$ is the disjoint union of four $K_4$'s.
\end{itemize}
For $q>4$, $\mathbb{F}_q$ is not an extremal ring and so the associated hyperbola graphs are connected, non-bipartite graphs; thus the pattern exhibited 
by $\mathbb{F}_2, \mathbb{F}_3$ and $\mathbb{F}_4$ does not continue.

Explicitly for $\mathbb{F}_5$, the hyperbola is given by $H=\{ (1,1), (2,3), (3,2), (4,4) \}$. Thus given $(x,y) \in \mathbb{F}_5^2$ it is clear there is a path 
from $(x,y)$ to all the $(x+n, y+n), n=0,1,2,3,4$ consisting of adding $(1,1) \in H$ repeatedly to $(x,y)$. On the other hand, adding $(2,3)$ or $(3,2)$ to 
$(x,y)$ raises or lowers the value of $y-x$ by one. From these facts it is easy to directly check that the hyperbola graph of $\mathbb{F}_5$ is connected.
In fact the $5$ vertices on the line $y-x=b$ for fixed $b$ form a cycle subgraph $C_5$. The hyperbola graph of $\mathbb{F}_5$ is obtained from the 
five cycle subgraphs for $b=0,1,2,3,4$ by joining each point in the cycle subgraph corresponding to the line $y-x=b$ to exactly one point in the cycle subgraph 
corresponding to the line $y-x=b+1$ and to exactly one point in the cycle subgraph for the line $y-x=b-1$.
\end{example}

\begin{example}
If $R_1, R_2$ are finite rings and $R=R_1 \times R_2$ is their direct product, the Chinese remainder theorem shows that $H=H_1 \times H_2$ 
where $H$ is the hyperbola of $R$ and $H_j$ is the hyperbola of $R_j$. The resulting hyperbola graph $G_H$ has $(x_1,y_1)$ adjacent to 
$(x_2,y_2)$ if and only if $x_1, x_2$ are adjacent in $G_{H_1}$ and $y_1, y_2$ are adjacent in $G_{H_2}$. Thus the adjacency matrix of 
$G_H$ is the tensor product of those for $G_{H_1}$ and $G_{H_2}$. Thus if $\lambda_1, \dots, \lambda_N$ is the spectrum of $G_{H_1}$ (listed with multiplicity) and 
$\mu_1, \dots, \mu_K$ is the spectrum of $G_{H_2}$ then $\lambda_i \mu_j, 1 \leq i \leq N, 1 \leq j \leq K$ is the spectrum of $G_H$.
\end{example}

Given a  graph, a random walk on the graph is a process where we start at some vertex and at each step move to an adjacent vertex 
in a manner where it is equally likely that we move to any adjacent vertex versus any other.

It is well known (see \cite{lovasz2007eigenvalues}) that the random walk on a connected, non bipartite, regular graph converges to the uniform distribution. This means that no matter where we start, after a large number of random steps, we are equally likely to be anywhere in the graph. More precisely, in the hyperbola graph for a non-extremal ring $R$, using the results in \cite{lovasz2007eigenvalues}, we have if 
$p_{ij}^t$ is the probability that starting at vertex $i$ we end up at vertex $j$ after $t$ steps in a random walk, then $p_{ij}^t$ satisfies
$$ \left| p_{ij}^t - \frac{1}{|R|^2} \right| \leq \left(\frac{C}{\sqrt{|R^*|}}\right)^t. $$ where $C$ is the Kloosterman-Salem number of the finite ring $R$.

A $d$-regular, connected, non-bipartite graph has good expansion properties if its spectral gap is large. In particular, 
if $d - \lambda_2 \geq 2\epsilon d$ then $G$ is an $\epsilon$-expander (see \cite{lovasz2007eigenvalues}). It follows that:

\begin{corollary}
If $R$ is a non-extremal ring with Kloosterman Salem number $C$, then the corresponding hyperbola graph is a $|R^*|$-regular, connected, non-bipartite 
simple graph and is an expander graph with expander ratio is 
$$\epsilon=\frac{1}{2}\left(1-\frac{C}{\sqrt{|R^*|}}\right).$$
\end{corollary}

\vskip.125in 

\begin{remark} Among expander graphs, the Ramanujan graphs are those with best spectral expansion. The hyperbola graph of a non-extremal ring $R$ is a Ramanujan graph if $\lambda_2 \leq 2\sqrt{|R^*|-1}$. This happens if and only if the Kloosterman-Salem number $C$ satisfies $C \leq 2\sqrt{1-\frac{1}{|R^*|}}$. If $R$ is an odd order ring, our results show that aside from a finite set of exceptions, this can only occur when $R$ is a field.

Further, a graph is Ramanujan if and only if its \emph{Ihara zeta function}\/ satisfied the ``Riemann Hypothesis'' \cite{murty2003ramanujan}.
The Ihara zeta function is defined for all graphs, and so it provides a zeta function associated to Kloosterman sums over general rings.
For Kloosterman sums over fields, the Ihara zeta function and the classical zeta function (\cite{iwaniec2004analytic} section 11.5) are closely related.
\end{remark}

The results of section \ref{geomcriterion} yield an upper bound on the independence number of hyperbola graphs.
\begin{proposition}
  \label{prop: independence number}
The independence number of the hyperbola graph of a finite ring $R$ with Kloosterman-Salem number $C_R$ is at most $\frac{C_R|R|^2}{\sqrt{|R^*|}}$.
\end{proposition}
\begin{proof}
  Let $E\subset R^2$ be an independent set.
This means that there are no solutions to $x-y\in H$ with $x$ and $y$ in $E$.
In the language of section \ref{geomcriterion}, this means that $n(E)=0$, hence by Theorem \ref{thm: analyticalbound}
\[
|E|\leq\frac{C_R|R|^2}{\sqrt{|R^*|}}.
\]
\end{proof}
This bound implies a lower bound on the chromatic number of hyperbola graphs.
\begin{proposition}
  \label{prop: chromatic number}
The chromatic number of the hyperbola graph of a finite ring $R$ with Kloosterman-Salem number $C_R$ is at least $\frac{\sqrt{|R^*|}}{C_R}$.
\end{proposition}
\begin{proof}
  Suppose that the hyperbola graph can be colored by $k$ colors so that no two adjacent vertices are the same color.
This partitions the vertex set $R^2$ into $k$ sets $E_1,\ldots,E_k$, where each $E_i$ is monochromatic.
Since vertices of the same color are not connected, each $E_i$ is an independent set, and so by Proposition \ref{prop: independence number}, we have
\[
|R^2|=\sum_{i=1}^k|E_i|\leq k\frac{C_R|R|^2}{\sqrt{|R^*|}}.
\]
Rearranging yields the desired lower bound on the number $k$ of colors required.
\end{proof}

\bibliography{salem_library}
\bibliographystyle{plain}

\enddocument

%% file: KSnum.tikz
\begin{tikzpicture}[xscale=3.0]
  \draw [thick, |->] (0,0) -- (4,0);
  \node[below] at (0,-0.2){0};
  \node[below right] at (4,0){$C$};
  \draw (1,-0.2) -- (1,0.2);
  \node[align=center, below] at (1,-0.2){$1$};
  \draw[->] (1,-1.1) -- (1,-.75);
  \node[align=center, below] at (1,-1.1){$\mathbb{F}_2^N$\\ \small Boolean \\ \small rings};
  \draw (1.4,-0.2) -- (1.4,0.2);
  \node[align=center, below] at (1.4,-0.2){$\sqrt 2$};
  \draw[->] (1.4,-1.1) -- (1.4,-.75);
  \node[align=center, below] at (1.4,-1.1){$\mathbb{F}_3$};
  \draw (2,-0.2) -- (2,0.2);
  \node[align=center, below] at (2,-0.2){$2$};
  \draw[->] (2,-1.7) -- (2,-.75);
  \node[align=center, below] at (2,-1.7){$C_F$\\ $\scriptstyle |F|\to\infty$};

  \draw (2.236,-0.2) -- (2.236,0.2);
  \node[align=center, below] at (2.236,-0.2){$\sqrt 5$};


  \draw (3,-0.2) -- (3,0.2);
  \node[align=center, below] at (3,-0.2){$\alpha$};
  \draw[<-, semithick] (2.6,-0.8) -- (3,-0.8);
  \node[align=left, below] at (2.9,-0.9){\small finite \# \\ \small non-field \\ \small odd rings} ;

\draw [decorate,decoration={brace,amplitude=6pt}, thick]
(0.05,0.15) -- (0.95,0.15);
\node [above] at (.5,.3){\small gap};

\draw [decorate,decoration={brace,amplitude=6pt}, thick]
(1.05,0.15) -- (1.35,0.15);
\node [above] at (1.2,.3){\small gap};


\draw [<-, semithick] (1.8,0.5) -- (2.236,0.5);
\node [above, align=center] at (2.02,0.6){\small all fields};

\end{tikzpicture}

%% file: salem_final.bbl
\begin{thebibliography}{10}

\bibitem{babai1989fourier}
L{\'a}szl{\'o} Babai.
\newblock The fourier transform and equations over finite abelian groups.
\newblock Lecture notes,
  \url{http://people.cs.uchicago.edu/~laci/reu02/fourier.pdf}, 1989.

\bibitem{chung1992quasi-random}
Fan R.~K. Chung and Ronald~L. Graham.
\newblock Quasi-random subsets of ${Z}/n{Z}$.
\newblock {\em Journal of Combinatorial Theory, Series A}, 61(1):64--86, 1992.

\bibitem{chung1989quasi-random}
Fan R.~K. Chung, Ronald~L. Graham, and Richard~M. Wilson.
\newblock Quasi-random graphs.
\newblock {\em Combinatorica}, 9(4):345--362, 1989.

\bibitem{deligne1974conjecture}
Pierre Deligne.
\newblock La conjecture de {W}eil. {I}.
\newblock {\em Inst. Hautes {\'E}tudes Sci. Publ. Math.}, 43:273--307, 1974.

\bibitem{deligne1980conjecture}
Pierre Deligne.
\newblock La conjecture de {W}eil. {II}.
\newblock {\em Inst. Hautes {\'E}tudes Sci. Publ. Math.}, 52:137--252, 1980.

\bibitem{iosevich2007erdos}
A.~Iosevich and M.~Rudnev.
\newblock Erd{\H o}s distance problem in vector spaces over finite fields.
\newblock {\em Trans. Amer. Math. Soc.}, 359(12):6127--6142 (electronic), 2007.

\bibitem{iwaniec2004analytic}
Henryk Iwaniec and Emmanuel Kowalski.
\newblock {\em Analytic number theory}, volume~53.
\newblock American Mathematical Society Providence, 2004.

\bibitem{katz1988gauss}
Nicholas~M. Katz.
\newblock {\em Gauss sums, {K}loosterman sums, and monodromy groups}, volume
  116 of {\em Annals of Mathematics Studies}.
\newblock Princeton University Press, Princeton, NJ, 1988.

\bibitem{kloosterman1927representation}
H.~D. Kloosterman.
\newblock On the representation of numbers in the form $ax^2+by^2+cz^2+dt^2$.
\newblock {\em Acta Mathematica}, 49(3-4):407--464, December 1927.

\bibitem{kowalskiexponential}
E.~Kowalski.
\newblock Exponential sums over finite fields, {I}: elementary methods.
\newblock \url{http://www.math.ethz.ch/~kowalski/exp-sums.pdf}.

\bibitem{kowalski2010aspects}
E.~Kowalski.
\newblock Some aspects and applications of the {R}iemann hypothesis over finite
  fields.
\newblock {\em Milan journal of mathematics}, 78(1):179--220, 2010.

\bibitem{lang2002algebra}
Serge Lang.
\newblock {\em Algebra}, volume 211 of {\em Graduate Texts in Mathematics}.
\newblock Springer-Verlag, New York, third edition, 2002.

\bibitem{lovasz2007eigenvalues}
L\'{a}szl\'{o} Lov\'{a}sz.
\newblock Eigenvalues of graphs.
\newblock \url{http://www.cs.elte.hu/~lovasz/eigenvals-x.pdf}, November 2007.

\bibitem{mazur2008finding}
Barry Mazur.
\newblock Finding meaning in error terms.
\newblock {\em Bull. Amer. Math. Soc. (N.S.)}, 45(2):185--228, 2008.

\bibitem{murty2003ramanujan}
M.~Ram Murty.
\newblock Ramanujan graphs.
\newblock {\em Journal-Ramanujan Mathematical Society}, 18(1):33--52, January
  2003.

\bibitem{titchmarsh1986theory}
E.~C. Titchmarsh.
\newblock {\em The theory of the {R}iemann zeta-function}.
\newblock The Clarendon Press, Oxford University Press, New York, second
  edition, 1986.

\bibitem{weil1948exponential}
Andr{{\'e}} Weil.
\newblock On some exponential sums.
\newblock {\em Proc. Nat. Acad. Sci. U. S. A.}, 34:204--207, 1948.

\end{thebibliography}
